\documentclass[reqno]{amsart}

\usepackage{amscd, amsmath, amssymb, amsfonts, mathtools,mathrsfs,hyperref}

\numberwithin{equation}{section}
\newtheorem{theorem}[equation]{Theorem}
\newtheorem{corollary}[equation]{Corollary}
\newtheorem{lemma}[equation]{Lemma}
\newtheorem{proposition}[equation]{Proposition}

\newtheorem*{theorem*}{Theorem}

\theoremstyle{remark}
\newtheorem{remark}[equation]{\bf Remark}
\theoremstyle{remark}

\theoremstyle{remark}
\newtheorem{definition}[equation]{\bf Definition}

\newcommand{\N}{\mathbb{N}}

\newcommand{\Q}{\mathbb{Q}}

\newcommand{\tens}[1]{%
  \mathbin{\mathop{\otimes}\displaylimits_{#1}}%
}

\newcommand{\Ufd}{\mathrm{UFD}}
\newcommand{\Rk}{\mathrm{rank}}
\newcommand{\f }{\mathrm{frac}}
\newcommand{\s}{\mathrm{Spec\  }}
\newcommand{\D}{\mathrm{dim}}

\newcommand{\Exp}{\mathrm{EXP}}

\newcommand{\Tr}{\mathrm{tr.deg}}
\newcommand{\mf}[1]{\mathfrak{#1}}
\newcommand{\ms}[1]{\mathscr{#1}}

\renewcommand{\o}[1]{\overline{#1}}

\title{Rigidity and triangularity of an exponential map }

\author[Sai Krishna]{P M S Sai Krishna} 
\address{\noindent P M S Sai Krishna, Department of Mathematics,  IIT Bombay,  Powai,  Mumbai 400076,  India} 
\email{\href{mailto:saikrishna183@gmail.com}{saikrishna183@gmail.com},  \hspace{3mm}\href{mailto:204099001@iitb.ac.in}{204099001@iitb.ac.in}}

\begin{document}
\date{}
\begin{abstract}
Let $k$ be a field of arbitrary characteristic, $A$ be a domain, and $K=\f (A)$. Then \begin{enumerate}
    \item All exponential maps of $k^{[3]}$ are rigid, and we give a necessary and sufficient condition for the triangularity of $\delta \in \Exp (k^{[3]})$. 

    \item If $\delta \in \Exp (A^{[3]})$ such that $\Rk (\delta)=\Rk (\delta_K)$, then $\delta$ is rigid and we give a necessary and sufficient condition for the triangularity of $\delta$.
\end{enumerate}
When $k$ is  of zero characteristic, $(1)$ is due to \cite{DD} and $(2)$ is due to \cite{KL}. 
\end{abstract}
 
\maketitle
\subjclass  2020 Mathematics Subject Classification:  {13A50, 13B25, 13N15, 14R20}

~  {polynomial ring,  exponential map, triangular exponential map, rank, rigidity}

\section{Introduction}
\emph{Throughout the paper, all rings are commutative with unity, $k$ is a field of arbitrary characteristic, $A \subset B$ are domains, and $K=\f (A)$. $A^{[n]}$ denotes the polynomial ring in $n$ variables over $A$ and $\Exp _A(B)$ denotes the set of all exponential maps on $B$. For $B=A^{[n]}$ we define a \emph{coordinate system} of $B$ over $A$ as a set of elements $\{X_1,\dots, X_n\}$ of $B$ such that $B=A[X_1,\dots,X_n]$ and $\Gamma(B)$ denotes the set of all coordinate systems of $B$ over $A$. Let $B=A^{[n]}$ and $\delta \in \Exp (B)$ with $\Rk (\delta)=r$. We denote by $\Gamma_{\delta}(B)$ the set of $\{X_1,\dots,X_n\} \in \Gamma(B)$ such that $A[X_1,\dots,X_{n-r}]\subset B^{\delta}$.}

Exponential maps are a generalization of locally nilpotent derivations. When $k$ is of zero characteristic, exponential maps are equivalent to locally nilpotent derivations. The notion of rank, rigidity, and triangularity of a locally nilpotent derivation of a polynomial ring was introduced in \cite{GF} and \cite{DD}, respectively. We define the same notions similarly for exponential maps of polynomial rings. 

In \cite{DD}, Daigle proved that all locally nilpotent derivations of $k^{[3]}$ are rigid and gave a necessary and sufficient condition for the triangularity of a locally nilpotent derivation. Under some additional hypothesis, the results in \cite{DD} were generalized in \cite{KL} to locally nilpotent derivations of $A^{[3]}$.  This paper will prove these results for exponential maps, which are stated below.
\begin{enumerate}
    \item \textbf{Theorem }(\ref{k3}). All exponential maps of $k^{[3]}$ are rigid. 

    \item \textbf{Theorem}(\ref{Dk2D}). Let $\delta \in \Exp _A(A^{[3]})$. Suppose $\Rk (\delta)=\Rk (\delta_K)$,where $K=\f (A)$ and $\delta_K\in \Exp (K^{[n]})$ is defined as in $(3)$ of (\ref{first principles}). Then $\delta$ is rigid. 

    \item \textbf{Theorem}(\ref{tri2}). Let $\delta \in \Exp (k^{[3]})$. Let $\{X,Y,Z\}\in \Gamma_{\delta}(k^{[3]})$. Then $\delta$ is triangular if and only if  $\delta$ is triangular over $k[X]$.

    \item \textbf{Theorem}(\ref{ATri}). Let $B=A^{[3]}$ and $\delta \in \Exp _A(B)$. Let $\{X,Y,Z\} \in \Gamma(B)$ such that $X\in B^{\delta}$. Suppose $\Rk (\delta)=\Rk (\delta_K)=2$, where $K=\f (A)$ and $\delta_K\in \Exp (K^{[n]})$ is defined as in $(3)$ of (\ref{first principles}) Then $\delta$ is triangular over $A$ if and only if $\delta$ is triangular over $A[X]$. 
\end{enumerate}

\section{Preliminaries}

We recall some definitions.
\begin{definition}
 Let $\delta: B\longrightarrow B^{[1]}$ be an $A$-algebra homomorphism. We denote it by $\delta_t:B\longrightarrow B[t]=B^{[1]}$ if we want to emphasize the indeterminate. We call $\delta$ an \emph{exponential map} on $B$ if
    \begin{enumerate}
        \item $\epsilon_0\delta_t$ is identity on $B$, where $\epsilon_0:B[t]\longrightarrow B$ is the evaluation map at $t=0.$

        \item $\delta_s\circ \delta_t=\delta_{s+t}$,  where $\delta_s$ is extended to a homomorphism $B[t]\longrightarrow B[s, t]$ by defining $\delta_s(t)=t$.
    \end{enumerate}
We denote the set of all exponential maps on $B$ by $\Exp _A(B)$. The set $B^{\delta}=\{x\in B| \delta(x)=x\}$ is called the \emph{ring of $\delta$-invariants} of $B$. We call $\delta \in \Exp _A(B)$ \emph{non-trivial} if $B^{\delta}\neq B$.
\end{definition}

\begin{remark}
Let $\delta \in \Exp _A(B)$ be such that  $\delta :B\longrightarrow B[t]=B^{[1]}$. We can express $\delta$ as follows $$\delta(x)=D_0(x)+D_1(x)t+D_2(x)t^2+\cdots+D_m(x)t^m $$ where $D_0, D_1, D_2, \dots $ are iterative derivatives associated with $\delta$ and $m\in \N$.  Note that $B^{\delta}=\bigcap _{i\geq 1}\mathrm{Ker}(D_i)$, $D_0=id$ and $D_i$ is $A$-linear for all $i$.
    
\end{remark}

\begin{definition}
Let $B$ $\delta \in \Exp _A(B)$ such that  $\delta :B\longrightarrow B[t]=B^{[1]}$.
\begin{enumerate}
    \item  Any $x\in B$ such that $\delta (x)$ has minimal positive $t$-degree is called a \emph{local slice} of $\delta$.

    \item  An element $x\in B$  is called a \emph{slice} if $x$ is a local slice of $\delta$ and $D_n(x)$ is a unit,  where $n= \deg_t(\delta(x))$.
    
    \item A subring $R$ of $B$ is said to be \emph{factorially closed} in $B$ if for any non-zero $x, y \in B$ such that $xy\in R$,  then $x, y \in R$. Note that if $R$ is factorially closed in $B$, then $R$ is algebraically closed in $B$.

    \item $A$ is an $\mathrm{HCF}$-ring (highest common factor ring) if the intersection of two principal ideals is again a principal ideal.  

    \item A $k$-domain $B$ is called \emph{geometrically factorial} if $B\tens{k} L$ is a $\Ufd$, where $L$ is any algebraic extension of $k$.   

    \item Let $B=A^{[2]}=A[X,Y]$. An element $W$ of $A[X,Y]$ is called a \emph{residual variable} if for every prime ideal $\mf{p}$ of $A$, $$\kappa(\mf{p})\tens{A}A[X,Y]=(\kappa(\mf{p})[\o{W}])^{[1]}$$  where $\kappa(\mf{p})=A_{\mf{p}}/\mf{p}A_{\mf{p}}=\f (A/\mf{p})$ and $\o{W}$ is the image of $W$ in $\kappa(\mf{p})\tens{A}A[X,Y]$.

    \item Let $B=A^{[n]}$. The \emph{rank} of $\delta$ is the least integer $r\geq 0$ for which there exists $\{X_1,\dots,X_n\} \in \Gamma(B)$ such that $A[X_1,\dots,X_{n-r}]\subset  B^{\delta}$. We denote it by $\Rk (\delta)$.

    \item Let $B=A^{[n]}$. We say $\delta$ is \emph{triangular} over $A$ if there exists a coordinate system $\{X_1,\dots,X_n\}$ such that 
    \begin{enumerate}
        \item $\delta(X_1)-X_1 \in tA[t]$.

        \item  $\delta(X_i)-X_i \in tA[X_1,\dots,X_{i-1}][t]$ for all $i\geq 2$.

        \item If $i$ is the smallest such that $\delta(X_i)\neq X_i$, then $X_i$ is a local slice of $\delta$.
    \end{enumerate}

    \item Let $B=A^{[n]}$ and $\Rk (\delta)=r$. Then $\delta$ is called \emph{rigid} if $A[X_1,\dots,X_{n-r}]=A[X_1',\dots,X_{n-r}']$, whenever $\{X_1,\dots,X_n\}, \{X_1',\dots,X_n'\} \in \Gamma_{\delta}(B)$. 

    \item  For $\alpha \in B^{\delta}$, define $\alpha \delta \in \Exp (B)$ as follows. $$\alpha \delta(x)=x+\alpha D_1(x)t+\alpha^2D_2(x)t^2+\cdots +\alpha^mD_m(x)t^m.$$. We note that $x$ is local slice of $\delta$ if and only if $x$ is a local slice of $\alpha \delta$ and when $B=A^{[n]}$, we have $\Rk(\delta)=\Rk(\alpha \delta)$.
\end{enumerate}
\end{definition}

We summarise below some useful properties of exponential maps \cite[Lemma $2.2$]{KW} and \cite[Lemma $2.1,  2.2$]{CM}.

\begin{proposition}\label{first principles}
Let $\delta \in \Exp _A(B)$ be non-trivial and $x$ be a local slice with $m=\deg_t\delta(x)$. Let $c=D_n(m)$.
\begin{enumerate}
    \item $D_i(x)\in B^{\delta}$ for all $i>0.$ 
    
    \item  $B^{\delta}$ is factorially closed in $B$ hence algebraically closed.

    \item Let $B=A^{[n]}$ and $S\subset B^{\delta}\setminus\{0\}$ be a multiplicative closed subset. Then $\delta$ extends to a non-trivial exponential map $S^{-1}\delta$ on $S^{-1}B$ defined as $S^{-1}\delta (\frac{b}{s})=\frac{\delta(b)}{s}$ for all $b\in B$ and $s\in S$. Moreover, $(S^{-1}B)^{S^{-1}\delta}=S^{-1}(B^{\delta})$. In particular with $S=A\setminus \{0\}$, $S^{-1}B=K^{[n]}$ and we denote $S^{-1}\delta$ by $\delta_K \in \Exp (K^{[n]})$. Note that if $x$ is a local slice of $\delta$, then $\frac{x}{1}$ is a local slice of $S^{-1}\delta$. 

    \item If $A=k$ is a field, then $B[c^{-1}]=B^{\delta}[c^{-1}][x]$ and   $\Tr _{B^{\delta}}B=1$.
\end{enumerate}
\end{proposition} 

The following result is a special case of $(4)$ of (\ref{first principles}). We record it separately for later use. 
\begin{corollary}[\emph{Slice Theorem}]\label{slice theorem}
Let $\delta$ be a non-trivial exponential map on a $k$-domain $B$ and $x$ be a slice. Then  $B=B^{\delta}[x]=(B^{\delta})^{[1]}$.
\end{corollary}

We require the following results related to exponential maps from \cite[Corollary $1.2$]{KW} and \cite[Corollary $3.3$]{SK}, respectively. 

\begin{proposition}\label{tr2}
    Let $A$ be an $\mathrm{HCF}$-ring, $B=A^{[2]}$ and $\delta \in \Exp _A(B)$ is non-trivial. Then there exists $h\in B\setminus A$ such that $B^{\delta}=A[h]=A^{[1]}$.
\end{proposition}

\begin{corollary}\label{gufd slice}
      Let $B$ be a $\Ufd$ over a field $k$ with $\Tr _kB=2$ and $\delta \in \Exp (B)$ is non-trivial. Suppose $B$ is geometrically factorial. Then $B=(B^{\delta})^{[1]}.$
\end{corollary}

We recall the following results from \cite[Proposition $4.8$ and $1.7$]{AEH}. 
\begin{proposition}\label{tr1}
    Let $A$ be an $\mathrm{HCF}$-ring and $B=A[X_1,\dots,X_n]=A^{[n]}$. Suppose $R$ is a ring and  $\Tr _AR=1$ such that $A\subset R\subset B.$ If $R$ is factorially closed in $B$, then $R= A^{[1]}$.
\end{proposition}

\begin{lemma}\label{unique}
    Let $R,S$ be subrings of a domain $A$ such that $A=R^{[n]}=S^{[n]}$. If $b\in S$ is such that $bS\cap R\neq 0$, then $b\in R$.
\end{lemma}

We will be using the following results related to variables of polynomial rings from \cite[Theorem $3.2$]{BD} and \cite[$2.4$]{PS}, respectively. 

\begin{proposition}\label{residual}
    Let $A$ be a Noetherian ring such that $A$ contains $\Q$ or $A_{\mathrm{red}}$ is seminormal. An element $W$ of $A[X,Y]= A^{[2]}$ is a variable over $A$ if and only if $W$ is a residual variable of $A[X,Y]$.   
\end{proposition}

\begin{theorem}\label{generic}
    If $f\in k[X,Y]=k^{[2]}$ is generically a line, i.e., $k(f)[X,Y]=k(f)^{[1]}$, then $f$ is a variable of $k[X,Y]$.
\end{theorem}

\section{Supporting results}
In this section, we prove some elementary results about exponential maps of $A^{[n]}$ of rank $1$ or $2$ and some basic results about triangular exponential maps. 

\begin{lemma}\label{basic}
  Let $B=A[X]= A^{[1]}$ and $\delta \in \Exp _A(B)$ be a non-trivial  exponential map from $A[X]\longrightarrow A[X][t]$. Then 
  \begin{enumerate}
      \item $B^{\delta}=A$.
      
      \item $X$ is a local slice of $\delta$.

      \item $\delta(X)-X \in tA[t]$.
  \end{enumerate}  
\end{lemma}

\begin{proof}
 We have that $A\subsetneq B^{\delta}$.  By $(2)$ of (\ref{first principles}), $B^{\delta}$ is algebraically closed in $B$ and hence $\Tr _{B^{\delta}}B\geq 1$. Moreover,  $A$ is also algebraically closed in $B$ and $\Tr _AB=1$. It follows that $A=B^{\delta}$. Since $\delta$ is non-trivial, we have $\delta(X)\neq X$ and let $m=\deg_t\delta(X)$. Let $P(X)\in A[X]$ and $n=\deg_X(P(X))$. Then $\deg_{t}(\delta(P(X))=m\cdot n.$  Thus, we can conclude that $X$ is a local slice of $\delta$. By $(1)$ of (\ref{first principles}), $D_i(X)\in A$ for all $i>0$, and hence $\delta(X)-X \in tA[t]$. 
\end{proof}

\begin{lemma}\label{rank1}
    Let $B=A^{[n]}$. Let  $\delta \in \Exp _A(B)$ of rank $1$ and  $\{X_1,\dots,X_n\}\in \Gamma_{\delta}(B)$. Then 
    \begin{enumerate}
        \item $B^{\delta}=A[X_1,\dots,X_{n-1}]$

        \item $\delta(X_n)-X_n\in tB^{\delta}[t]$
    \end{enumerate}
\end{lemma}

\begin{proof}
        This follows by applying (\ref{rank1}) to $\delta \in \Exp_C(B)$, where 
        $C=A[X_1,\dots,X_{n-1}]$ and $B=C[X_n]$.
\end{proof}
The following result is a direct consequence of (\ref{basic}) and (\ref{rank1}). 
\begin{corollary}\label{rk1}
    All rank $1$ exponential maps of $A^{[n]}$ are rigid and triangular.  
\end{corollary}

The following result gives a description of ring of invariants of $\Rk$ $2$ exponential maps of $k^{[n]}$. When $k$ is of zero characteristic, this result is due to \cite[corollary $3.2$]{DF}).

\begin{proposition}\label{rank2}
    Suppose $A$ is a $\mathrm{HCF}$-ring   with $K=\f (A)$. Let $B=A^{[n]}$ and $\delta \in \Exp _A(B)$.  Suppose $\Rk (\delta)=2$ and  $\{X_1,\cdots,X_n\}\in \Gamma_{\delta}(B)$. Then $B^{\delta}=A^{[n-1]}=A[X_1,\cdots,X_{n-2}][f]$, where $f\in B$ and $f$ is a variable of $K(X_1,\cdots,X_{n-2})[X_{n-1},X_n]$.
\end{proposition}

\begin{proof}
    Applying (\ref{tr2}) with $R=A[X_1,\cdots,X_{n-2}]$ , it follows that $B^{\delta}=A[X_1,\cdots,X_{n-2}][f]$.
    
    Let $S=A[X_1,\cdots,X_{n-2}]\setminus \{0\}$. By $(3)$ of (\ref{first principles}), $\delta$ induces a non-trivial exponential map $(S^{-1}\delta)$ on  $K(X_1,\cdots,X_{n-2})[X_{n-1},X_n]$ with ring of invariants $K(X_1,\cdots,X_{n-2})[f]$. By (\ref{gufd slice}), the result follows. 
\end{proof}

\begin{corollary}\label{slice}
    Let $B=k^{[3]}$ and $\delta \in \Exp (B)$. 
     \begin{enumerate}
         \item  If $\delta$ has  a slice,  then $B^{\delta}=k^{[2]}$ and $\Rk (\delta)=1$.

         \item  If $\delta$ is triangular with respect to $\{X,Y,Z\}\in \Gamma(B)$ and $\Rk(\delta)=2$, then $\delta(X)=X$. In particular, any triangular exponential map of $k^{[3]}$ has rank at most $2$.
     \end{enumerate}
\end{corollary}

\begin{proof}
    \emph{$(1):$} By $(4)$ of $\ref{first principles}$, $B=(B^{\delta})^{[1]}$. Since $k^{[2]}$ is cancellative \cite[Theorem $2.7$]{BG}, we get $B^{\delta}=k^{[2]}$ and since $(B^{\delta})^{[1]}=k^{[3]}$, it follows that $\Rk (\delta)=1$.

    \emph{$(2):$} As $\delta$ is triangular, hence  $\delta(X) \in X+tk[t]$.  If $\delta(X)\neq X$, then $X$ is a slice, and by the previous part we get $\Rk(\delta)=1$ which is a contradiction. This proves that $\delta(X)=X$.
\end{proof}

\section{Rigidity of exponential maps}
In this section, we show that all exponential maps of $k^{[3]}$ are rigid and extend this result to exponential maps of $A^{[n]}$ under some assumptions. The following lemma is proved by Daigle \cite[Lemma $2.4$]{DD} in zero characteristic for locally nilpotent derivations of $k^{[3]}$. We use similar arguments for exponential maps. 
\begin{lemma}\label{technical}
    Let $B=k^{[3]}$ and $\delta \in \Exp(B)$ with $\Rk (\delta)=2$ and $A=B^{\delta}$. Define 
     $$\ms{R}(\delta)=\{\alpha \in \s A|\kappa(\alpha) \tens{A} B \neq \kappa(\alpha)^{[1]}\}$$ where $\kappa(\alpha)=A_{\alpha}/\alpha A_{\alpha}$ and let $\o{\ms{R}(\delta)}$ be the closure of $\ms{R}(\delta)$ in $\s A$. Then
     \begin{enumerate}
         \item $\D  ( \o{\ms{R}(\delta)})=1$.

         \item Let $\eta$ be an irreducible element of $A$ such that $V(\eta)$ is an irreducible component of $\o{\ms{R}(\delta)}$. Then, for every $\{X,Y,Z\}\in \Gamma_{\delta}(B)$, $k[X]$ is the integral closure of $k[\eta]$ in $B$. 
     \end{enumerate}
\end{lemma}

\begin{proof}

   \emph{(1): } Let $\{X,Y,Z\} \in \Gamma_{\delta}(B)$, where $B=k^{[3]}=k[X,Y,Z]$. By (\ref{rank2}), \begin{equation} \label{one}
A=k[X,f],
\end{equation}
 for some $f\in B$ and $f$ is a variable of $k(X)[Y,Z]$, which means that there exists $g\in k(X)[Y,Z]$ such that \begin{equation}\label{two}
      k(X)[Y,Z]=k(X)[f,g]. 
 \end{equation}
Since $\Rk (\delta)=2$, \begin{equation}\label{three}
     \textit{f is not a variable of B over } k[X].
\end{equation} 

By (\ref{two}), there exists $h\in k[X]\setminus \{0\}$ such that $B_h=(k[X])_h[f,g]$, which is same as $B_h=A_h^{[1]}$ and hence $\ms{R}(\delta)\subseteq V(h)$. Thus $\o{\ms{R}(\delta)}\subset V(h)\subset \s A$ and it follows that $\D (\o{\ms{R}(\delta)})\leq 1$. We will show that $\D (\o{\ms{R}(\delta)})>0$.

Since $k[X]$ is a Noetherian domain and seminormal, by (\ref{three}) and (\ref{residual}), it follows that $f$ is not a residual variable of $B$ over $k[X]$. Thus, there exists a prime ideal $\mf{p}\in \s k[X]$ such that \begin{equation}\label{four}
    \o{f} \textit{ is not a variable of } L[\o{Y},\o{Z}], 
\end{equation}
where $L=\kappa(\mf{p})=\f (k[X]/\mf{p})$. By (\ref{two}), $\mf{p}$ must be a non-zero prime ideal and hence $L=\kappa(\mf{p})=k[X]/\mf{p}$. Let $\alpha=\mf{p}A=\mf{p}[f] \in \s A$. Similarly, $\mf{p}B=\mf{p}[Y][Z]$ and hence $\mf{p}B \in \s B.$

\emph{claim: $\alpha \in \ms{R}(\delta)$}

We have $k[X]\xhookrightarrow
{}A \xhookrightarrow{} B$. Since $\mf{p}$ is a principal ideal of $k[X]$, $\alpha$ and $\mf{p}B$ are also principal ideals. We have $\alpha$ contracts to $\mf{p}$ and since $A$ is factorially closed in $B$, thus $\mf{p}B$ contracts to $\alpha$. Hence, we have 
$$k[x]/\mf{p} \xhookrightarrow{}A/\alpha \xhookrightarrow{} B/\mf{p}B,$$
where $B/\mf{p}B=L[\o{Y},\o{Z}]$. Since $\alpha=\mf{p}[f]$, $A/\alpha= L[\o{f}]$ and thus $\kappa(\alpha)=A_{\alpha}/\alpha A_{\alpha}=\f (A/\alpha)=L(\o{f})  \subset \f (L[\o{Y},\o{Z}])$. Similarly, we have $\kappa(\alpha)\tens{A}B=\f (A/\alpha)\tens{A}B=S^{-1}A\tens{A}A/\alpha\tens{A}B=S^{-1}A\tens{A}(B/\mf{p} B)=\o{S}^{-1}L[\o{Y},\o{Z}]$, where $S=A\setminus \alpha$ and $\o{S}$ is the image of $S$ in $L[\o{Y},\o{Z}]$ given by $\o{S}=L[\o{f}]\setminus 0$. We get  $\kappa(\alpha)\tens{A}B=L(\o{f})[\o{Y},\o{Z}] \subset \f (L[\o{Y},\o{Z}])$.

Thus, if $\kappa(\alpha) \tens{A}B=\kappa(\alpha)^{[1]}$, then $L(\o{f})[\o{Y},\o{Z}]=(L(\o{f}))^{[1]}$ when identified inside $\f  (L[\o{Y},\o{Z}])$. It follows that $\o{f}$ is generically a line of $L[\o{Y},\o{Z]}$. By (\ref{generic}), $\o{f}$ is a variable of $L[\o{Y},\o{Z}]$, which contradicts (\ref{four}). Hence, $\alpha \in \ms{R}(\delta)$. Since $\alpha$ is not a maximal ideal of $A$, it follows that $\D (\o{\ms{R}(\delta)})>0$ and hence $\D (\o{\ms{R}(\delta)})=1$.

\emph{(2): } For this part, let $\eta \in A$ be an irreducible element of $A$ such that $V(\eta)$ is an irreducible component of $\o{\ms{R}(\delta)}$. By the previous part, we have $\o{\ms{R}(\delta)} \subset V(h)$, it follows that $V(\eta) \subset V(h)$ and hence $\eta$ is a prime factor of $h$ and in particular, $\eta \in k[X]$. Thus, $k[\eta] \subset k[X]$ is an integral extension, and it follows that $k[X]$ is the integral closure of $k[\eta]$ in $B$.
\end{proof}
The following theorem is proved in \cite[Theorem $2.5$]{DD} when $k$ is of zero characteristic. 
\begin{theorem}\label{k3}
    All exponential maps of $k^{[3]}$ are rigid. 
\end{theorem}

\begin{proof}
    Let $B=k^{[3]}$. By (\ref{rk1}), all exponential maps of $B$ of rank $1$ are rigid. Let $\delta \in \Exp (B)$ such that $\Rk (\delta)=2$ and $A=B^{\delta}$. Let $\{X,Y,Z\}, \{X',Y',Z'\} \in \Gamma_{\delta}(B)$. We can choose $\eta \in A$ such that $V(\eta)$ is an irreducible component of $\o{\ms{R}(\delta)}$. This is possible since $\D (\o{\ms{R}(\delta))})=1$ and height one prime ideals of $A=k^{[2]}$ are principal. By (\ref{technical}), $k[X]$ and $k[X']$ are both equal to the integral closure of $k[\eta]$ in $B$. It follows that $k[X]=k[X']$. Thus, all rank $2$ exponential maps of $k^{[3]}$ are rigid. By definition, all rank $3$ exponential maps of $k^{[3]}$ are rigid, and we are done. 
\end{proof}

The following theorem is proved in \cite[Theorem $3.1$]{KL} when $k$ is of zero characteristic. 
\begin{proposition}\label{Dk2D}
    Let $\delta \in \Exp _A(A^{[n]})$. Assume that $\Rk (\delta)=\Rk (\delta_K)$, where $K=\f (A)$ and $\delta_K\in \Exp (K^{[n]})$ is defined as in $(3)$ of (\ref{first principles}). If $\delta_K$ is rigid, then  $\delta$ is rigid. 
\end{proposition}

\begin{proof}
    This follows by using  $(\ref{unique})$ and following the same proof as in \cite[Theorem $3.1$]{KL}.
\end{proof}
The following result follows from (\ref{k3}) and $(\ref{Dk2D})$.
\begin{theorem}\label{use}
    Let $\delta \in \Exp _A(A^{[3]})$. Suppose $\Rk (\delta)=\Rk (\delta_K)$, then $\delta$ is rigid. 
\end{theorem}

\section{Triangularity of exponential maps}
In this section, we give a necessary and sufficient condition for the triangularity of an exponential map of $k^{[3]}$ and extend this result to exponential maps of $A^{[n]}$ under some assumptions. The following result is proved in \cite[Corollary $3.3$]{DD} when $k$ is of zero characteristic. 
\begin{proposition}\label{tri1}
    Let $B=k^{[3]}$ and $\delta \in \Exp (B)$ such that $\Rk (\delta)=2$. Suppose $\delta$ is triangular with respect to $\{X,Y,Z\}\in \Gamma(B)$. Then for all $\alpha \in B^{\delta}$, we have $\alpha \delta $ is triangular if and only if $  \alpha \in k[X]$.
\end{proposition}

\begin{proof} 
     Since $\delta$ is triangular of rank $2$, by (\ref{slice}), we have $\delta(X)=X$. Suppose $\alpha \in k[X]$, then  $\alpha \delta$ is also triangular with respect to $\{X,Y,Z\}\in \Gamma(B)$. Now, suppose $\alpha \delta$  is triangular with respect to $\{ X',Y',Z' \}$.  Since $\Rk(\alpha \delta)=\Rk(\delta)=2$, by (\ref{slice}),  $\alpha \delta(X')=X'$ and hence $\delta(X')=X'$. Also, we have $\delta(X)=X$. Thus, $\{X,Y,Z\}, \{ X',Y',Z' \} \in \Gamma_{\delta}(B)$. By (\ref{k3}), $\delta$ is rigid and hence $k[X']=k[X]$. Since $\alpha \delta (Y')\in Y'+tK[X'][t]$ and $\alpha \delta(Y')\neq Y'$ as $\Rk(\delta)=\Rk(\alpha\delta)=2$, it follows that $\alpha \in k[X']$ and hence $\alpha \in K[X]$.
\end{proof}

The following result is proved in \cite[Corollary $3.4$]{DD} when $k$ is of zero characteristic. 
\begin{proposition}\label{tri2}
        Let $B=k^{[3]}$ and $\delta \in \Exp (B)$. Let $\{X,Y,Z\}\in \Gamma_{\delta}(B)$. Then $\delta$ is triangular if and only if  $\delta$ is triangular over $k[X]$. 
\end{proposition}

\begin{proof}
  $\impliedby$ is clear. Let $\delta$ be triangular with respect to $\{ X',Y',Z' \}\in \Gamma(B)$. If $\Rk(\delta)=1$, the result follows by (\ref{rank1}). Let $\Rk (\delta)=2$. By (\ref{slice}), $\delta(X')=X'$.  By (\ref{k3}), $\delta$ is rigid and we have $\{ X',Y',Z' \}\in \Gamma_{\delta}(B)$ which gives $k[X]=k[X']$ and it follows that $\delta$ is trianglular over $k[X]$ with the coordinate system $\{Y',Z'\}$ over $k[X]$.
\end{proof}

The following two results are proved in \cite[Proposition $4.2$ and $4.3$]{KL} when $k$ is of zero characteristic. 
\begin{proposition}
    
        Let $A$ be a Noetherian seminormal domain and $B=A^{[3]}$. Let $\delta \in \Exp _A(B)$ such that $\Rk (\delta)=1$. Let $\{X,Y,Z\}\in \Gamma(B)$ such that $X\in B^{\delta}$. Then $\delta $ is triangular over $A[X]$. 
\end{proposition}

\begin{proof}
    Since $\Rk (\delta)=1$, there exists $\{ X',Y',Z' \}\in \Gamma(B)$ such that $X',Y' \in B^{\delta}$. By (\ref{rank1}), we have $B^{\delta}=A[X',Y']$ and $\delta(Z')\in Z'+tA[X',Y'][t]$. We claim that $X$ is a variable of $A[X', Y']$ over $A$. By $(\ref{residual})$, it is enough to show that $X$ is a residual variable of $A[X', Y']$. Let $\mf{p}$ be a prime ideal of $A$. By $(3)$ of (\ref{first principles}), $\delta$ extends to an exponential map on $A_{\mf{p}}[X,Y,Z]$ with ring of invariants $A_{\mf{p}}[X',Y']$ and going modulo $\mf{p}A_{\mf{p}}$, we have an exponential map $\o{\delta}$ on $\kappa(\mf{p})[\o{X},\o{Y},\o{Z}]=A_{\mf{p}}/\mf{p}A_{\mf{p}}[\o{X},\o{Y},\o{Z}]$ with ring of invariants $(\kappa(\mf{p})[\o{X},\o{Y},\o{Z}])^{\o{\delta}}= \kappa(\mf{p})[\o{X'},\o{Y'}]$. 
    
    Since $X\in B^{\delta}=A[X',Y']$, we have that $\kappa(\mf{p})[\o{X}] \subset (\kappa(\mf{p})[\o{X},\o{Y},\o{Z}])^{\o{\delta}}\subset \kappa(\mf{p})[\o{X},\o{Y},\o{Z}]$ and by (\ref{tr1}), we get $(\kappa(\mf{p})[\o{X},\o{Y},\o{Z}])^{\o{\delta}}= (\kappa(\mf{p})[\o{X}])^{[1]}$. Thus, $\kappa(\mf{p})[\o{X'},\o{Y'}]= (\kappa(\mf{p})[\o{X}])^{[1]}$ and hence $X$ is a residual variable of $A[X',Y']$ and thus $X$ is a variable of $A[X',Y']$ over $A$. We have that $A[X',Y']=A[X,P]$ for some $P\in A[X',Y']$ and hence $B=A[X,P,Z']$. We note that $\delta(Z') \in Z'+tA[X',Y'][t]=Z'+tA[X,P][t]$ and $\delta(P)=P$. We have that $B=(A[X])^{[2]}$ and there exists a coordinate system $\{P,Z'\}$ over $A[X]$ such that $\delta(P)=P \in P+A[X][t]$ and $\delta(Z')\in Z'+tA[X,P][t]$ which shows that $\delta$ is triangular over $A[X]$.
\end{proof}

\begin{proposition}\label{ATri}
    
    Let $B=A^{[3]}$ and $\delta \in \Exp _A(B)$. Let $\{X,Y,Z\} \in \Gamma(B)$ such that $X\in B^{\delta}$. Suppose $\Rk (\delta)=\Rk (\delta_K)=2$, where $K=\f (A)$ and $\delta_K\in \Exp (K^{[n]})$ is defined as in $(3)$ of (\ref{first principles}) Then $\delta$ is triangular over $A$ if and only if $\delta$ is triangular over $A[X]$. 
\end{proposition}

\begin{proof}
    $\impliedby $ is clear. Suppose $\delta$ is triangular over $A$. Then there exists $\{ X',Y',Z' \} \in \Gamma(B)$ such that $\delta(X')\in X'+tA[t], \delta(Y')\in Y'+tA[X'][t]$ and $\delta(Z') \in Z'+ tA[X',Y'][t]$. We claim that $\delta(X')=X'$. If $\delta(X')=X'+f(t)$, where $f\in A[t]$ and $f\neq 0$. Then $X'$ is a local slice of $\delta$ and it follows that $X'$ is a slice of $\delta_K$ and by (\ref{slice}), we have $\Rk (\delta_K)=1$ which is a contradiction. Thus, $\delta(X')=X'$. By (\ref{use}) $\delta$ is rigid of rank $2$ and $\{X,Y,Z\} \in \Gamma_{\delta}(B)$ and $\{ X',Y',Z' \}\in  \Gamma_{\delta}(B)$ gives that $A[X']=A[X]$. By considering the coordinate system $\{Y', Z'\}$ over $A[X]$, it follows that $\delta$ is triangular over $A[X]$.
\end{proof}

\noindent
{\bf Acknowledgement:} The author thanks Manoj K. Keshari for reviewing earlier drafts and suggesting improvements. The author is supported by the Prime Minister's Research Fellowship (ID1301165).

\bibliographystyle{abbrv}
\bibliography{refs} 

\begin{thebibliography}{10}

\bibitem{AEH}
S.~S. Abhyankar, P.~Eakin, and W.~Heinzer.
\newblock On the uniqueness of the coefficient ring in a polynomial ring.
\newblock {\em J. Algebra}, 23:310--342, 1972.

\bibitem{BD}
S.~M. Bhatwadekar and A.~K. Dutta.
\newblock On residual variables and stably polynomial algebras.
\newblock {\em Comm. Algebra}, 21(2):635--645, 1993.

\bibitem{BG}
S.~M. Bhatwadekar and N.~Gupta.
\newblock A note on the cancellation property of {$k[X,Y]$}.
\newblock {\em J. Algebra Appl.}, 14(9):1540007, 5, 2015.

\bibitem{CM}
A.~J. Crachiola and L.~G. Makar-Limanov.
\newblock An algebraic proof of a cancellation theorem for surfaces.
\newblock {\em J. Algebra}, 320(8):3113--3119, 2008.

\bibitem{DD}
D.~Daigle.
\newblock A necessary and sufficient condition for triangulability of
  derivations of {$k[X, Y,Z]$}.
\newblock {\em J. Pure Appl. Algebra}, 113(3):297--305, 1996.

\bibitem{DF}
D.~Daigle and G.~Freudenburg.
\newblock Locally nilpotent derivations over a {UFD} and an application to rank
  two locally nilpotent derivations of {$k[X_1,\cdots,X_n]$}.
\newblock {\em J. Algebra}, 204(2):353--371, 1998.

\bibitem{GF}
G.~Freudenburg.
\newblock Triangulability criteria for additive group actions on affine space.
\newblock {\em J. Pure Appl. Algebra}, 105(3):267--275, 1995.

\bibitem{KL}
M.~K. Keshari and S.~A. Lokhande.
\newblock A note on rigidity and triangulability of a derivation.
\newblock {\em J. Commut. Algebra}, 6(1):95--100, 2014.

\bibitem{KW}
H.~Kojima and N.~Wada.
\newblock Kernels of higher derivations in {$R[x,y]$}.
\newblock {\em Comm. Algebra}, 39(5):1577--1582, 2011.

\bibitem{SK}
S.~Kuroda.
\newblock A generalization of {N}akai's theorem on locally finite iterative
  higher derivations.
\newblock {\em Osaka J. Math.}, 54(2):335--341, 2017.

\bibitem{PS}
P.~Russell and A.~Sathaye.
\newblock On finding and cancelling variables in {$k[X,\,Y,\,Z]$}.
\newblock {\em J. Algebra}, 57(1):151--166, 1979.

\end{thebibliography}
\end{document}